\RequirePackage{fix-cm}
\documentclass[smallextended]{svjour3}       
\smartqed  
\usepackage{graphicx}

\def\C{\mathbb{C}}
\def\c2{\mathbb{C}^2}
\def\R{\mathbb{R}}

\def\N{\mathbb{N}}

\def\P{\mathbb{P}}

\def\1{\bold{1}}

\def\a{\alpha}

\def\e{\varepsilon}

\def\P2{\mathbf{P}^2}

\def\C2{\mathbf{C}^2}
\def\C3{\mathbf{C}^3}

\def\moins{^{-1}}
\def\dis{\displaystyle}

\def\C{\mathbb{C}}
\def\R{\mathbb{R}}

\def\P{\mathbb{P}}

\def\N{\mathbb{N}}
\def\P2{\mathbb{P}^2}

\def\1{\mathbf{1}}

\def\a{\alpha}

\def\e{\varepsilon}

\usepackage{amsmath}

\usepackage[T1]{fontenc}
\usepackage[latin1]{inputenc}
\usepackage[english]{babel}
\usepackage{latexsym}
\usepackage{amsfonts}
\usepackage{amssymb}
\usepackage{graphicx}
\usepackage{dsfont}

\usepackage[all]{xy}
\xyoption{matrix}
\xyoption{arrow}
\usepackage{version}

\usepackage{graphics,epsfig,color}

\begin{document}

\title{Dynamics of non cohomologically hyperbolic automorphisms of $\mathbb{C}^3$
}


\author{Fr\'ed\'eric Protin 
}


\institute{F. Protin \at
              INSA de Toulouse\\
              \email{fredprotin@yahoo.fr}           
}

\date{Received: date / Accepted: date}

\maketitle

\begin{abstract}
We study the dynamics of a family of non cohomologically hyperbolic automorphisms $f$ of $\mathbb{C}^3$. We construct a compactification $X$ of $\mathbb{C}^3$ where their extensions are algebraically stable. We finally construct canonical invariant closed positive $(1,1)$-currents for $f^*$, $f_*$ and we study several of their properties. Moreover, we study the well defined current $T_f \wedge T_{f\moins}$ and the dynamics of $f$ on its support. Then we construct an invariant positive measure $T_f \wedge T_{f\moins}\wedge \phi_{\infty}$, where $\phi_{\infty}$ is a function defined on the support of $T_f \wedge T_{f\moins}$. We prove that the support of this measure is compact and pluripolar. We prove also that this measure is canonical, in some sense that will be precised.
\end{abstract}

\section{Introduction}
\label{intro}

Our purpose is to study the dynamics of a family of non cohomologically hyperbolic automorphisms (in the sense of \cite{Gs}) of $ \mathbf{C}^3$. The dynamics of such polynomial automorphisms is less understood than that of cohomologically hyperbolic automorphisms. More precisely, we study the following family:
\begin{equation}\label{maegawa}
 \dis f:(x,y,z)\in \mathbf{C}^3\longmapsto\left(y, z, yz+by+cz+dx+e\right)\in \mathbf{C}^3,
\end{equation}
where $b,c,d,e$ are complex parameters, with $d\neq 0$. These are the first examples of non cohomologically hyperbolic automorphisms with a rich dynamics. This family has already been studied by Maegawa \cite{M}, but our approach is quite different. We will construct a dynamical system $(X,f)$, where $X$ is a "nice compactification" of $\mathbb{C}^3$.
\smallskip

In Section 1 we construct the space $X$, a "nice compactification" of $\mathbb{C}^3$ useful for the study of $f$. We then obtain in Section 2 a canonical invariant class in $H^{1,1}(X,\mathbb{R})$. This is the same for $f$ and $f\moins$. We prove that this class is K\"ahler. This allows us to construct, in Section 3, positive invariant currents $T_f$ and $T_{f\moins}$ for $f$ and $f\moins$. We show that their potentials are continuous in $\mathbb{C}^3$ and pluriharmonic on the basin of attraction of infinity. In Section 4 we study the current $T_f\wedge T_{f\moins}$ and give a filtration of the dynamics of $f$ on $\dis\text{supp}(T_f\wedge T_{f\moins})$. 
 
\section{Good dynamical compactification for $f$ and $f^{-1}$}
\label{sec:1}

Up to conjugating $f$ by a suitable affine automorphism, we can assume without loss of generality that $ e = 0 $. We then have
 \begin{equation}\label{eq.maegawamoins} 
 f^{-1}(x,y,z)=\left(\frac{z-xy-bx-cy}{d}, x, y\right).
\end{equation}

If $\tau(x,y,z)=(z,y,x)$, then $\tau^{-1}=\tau$ and $\tau f \tau=(xy+by+cx+dz,x,y)$, which has the same form as $f^{-1}$. Therefore, the general behaviour of $f$ near infinity can be obtained from that of $f^{-1}$ by symmetry. Note that the jacobian of $f$ is constant and equal to $|d|$.
\smallskip

\paragraph{{\bf Extension of $f$ to $\mathbb{P}^3$.}} 

Let us denote by $[x,y,z,t]$ the homogeneous coordinates for $\mathbb{P}^3$, and by $H_{\infty}:=(t=0)$ the hyperplane at infinity. Written in homogeneous coordinates, the extension of $f$ to $\mathbb{P}^3=\mathbb{C}^3\sqcup H_{\infty}$, 
still denoted by $f$, is:
 \begin{equation}\label{homogene}
  f=[yt: zt: yz+byt+czt+dxt: t^2].
  \end{equation} 
  
  We denote by $I_f$ the set of indeterminacy of $f$, i.e. the points on $H_{\infty}$ where $f$ does not extend continuously.
  
  \smallskip
  
 Recall that a rational transformation $g:X \rightarrow X$ of a projective manifold $X$
  is {\it algebraically stable} (AS) if there is no hypersurface mapped by some iterate $g^n$ in the set $ I^+:=\cup_{n}I_{g^n}$ (\cite{FS}). 
  This is equivalent to the fact that $g$ induces a linear action $g^*$ on the cohomology vector space $H^{1,1}(X,\R)$
  which is compatible with the dynamics, in the sense that
  $(g^n)^*=(g^*)^n$ for all $n \in \N$. This condition is crucial in order to construct interesting invariant currents.
  
Note that $f$ is not AS in $\mathbb{P}^3$ since it contracts the hyperplane at infinity $H_{\infty}=(t=0)$ onto the point of indeterminacy $[0,0,1,0 ]\in I_f$. Setting 
 $$
 L:=(Y=T=0),\text{ }p^+:=[1,0,0,0],\text{ }p^-:=[0,0,1,0],
 $$ 
 $$
L':=(Z=T=0), L'':=(X=T=0), 
$$ 
we have the following more accurate description of the behavior of $f$ at infinity in $\mathbb{P}^3$.

\begin{lemma}\label{indeterminate}
The indeterminacy loci of $f$, respectively $f^{-1}$ in $\mathbb{P}^3$ are 
$$
I_f= L\cup L', \; I_{f^{-1}}=L\cup L'',
\; I_{f^2}=I_f \text{ and } I_{f^{-2}}=I_{f^{-1}}.
$$
Let $x\in L\setminus \{  p^+\}  $ and $x'\in L\setminus \{  p^-\} $. Then $$f(x)=L''\text{ and }f^{-1}(x')=L'.$$
Let $y\in L'\setminus \{  p^+\} $ and $y'\in L''\setminus \{  p^-\} $. Then $$f(y)=L\text{ and }f^{-1}(y')=L.$$
We have $f(p^+)=H_{\infty}$ and $f^{-1}(p^-)=H_{\infty}$.
\end{lemma}

\begin{proof}
By symmetry it is sufficient to prove the result for $f$. Let $x\in L\setminus p^+$ and a sequence $(z_n)\in \mathbb{P}^3\setminus I_f$ such that $z_n\rightarrow x$. 
Then $\left(f(z_n)\right)$ accumulates on $(X=T=0)$ by (\ref{homogene}). In a similar way we compute $f(y)=L$ when $y\in (Z=T=0)\setminus p^+$.

Now let $(z_n)\rightarrow p^+$; this means that $ \frac{Y}{X}\rightarrow 0$, $ \frac{Z}{X}\rightarrow 0$, $ \frac{T}{X}\rightarrow 0$. Moreover suppose that $z_n$ belongs to the cubic hypersurface $S_e$ defined by\begin{equation}\label{S}\frac{z}{t}+d\frac{x}{y}=e\end{equation}\noindent for $e\in \mathbf{C}$ (note that $p^+\in S_e$). In this case, we can replace the homogeneous representation (\ref{homogene}) by $ f=\left[1, \frac{Z}{Y}, b+c\frac{Z}{Y}+e, \frac{T}{Y}\right]$. For all $a', b'\in\mathbf{C}$ there exists $e\in\mathbf{C}$ and $(z_n)\rightarrow p^+$ such that $(z_n)\in S$ and $\frac{Z}{Y}\rightarrow a'$, $b+c\frac{Z}{Y}+e\rightarrow b'$. 

Let us now compute the last coordinate of $f$ in this case. We have $\frac{T}{Y}=\frac{T}{Z}\frac{Z}{Y}$. Since $\frac{Z}{Y}\rightarrow a'$ and $\frac{T}{Z}\rightarrow 0$ according to (\ref{S}) (recall that $ \frac{Y}{X}\rightarrow 0$), we have $\frac{T}{Y}\rightarrow 0$. Finally, for every $q\in H_{\infty}$ there exists a sequence converging to $p^+$ that is mapped into a sequence converging to $q$, hence $f(p^+)=H_{\infty}$.\qed
\end{proof}

We note that the hyperplane $H_{\infty}\setminus I_f^+$ is contracted into the point $p^-\in I_f$ by $f$ and into the point $p^+=[1: 0: 0: 0]\in I_{f^{-1}}$ by $f^{-1}$: thus neither $f$ nor $f^{-1}$ is algebraically stable on $ \mathbf{P}^3 $. Figure 1 summarizes the image by $f$ of the points of indeterminacy; $I_f$ is represented by straight lines and the action by $f$ by curved arrows. The figure for $f^{-1}$ is obtained by symmetry.

\begin{figure}\label{fig1}
\begin{center}
\includegraphics[width=13cm, height=8cm]{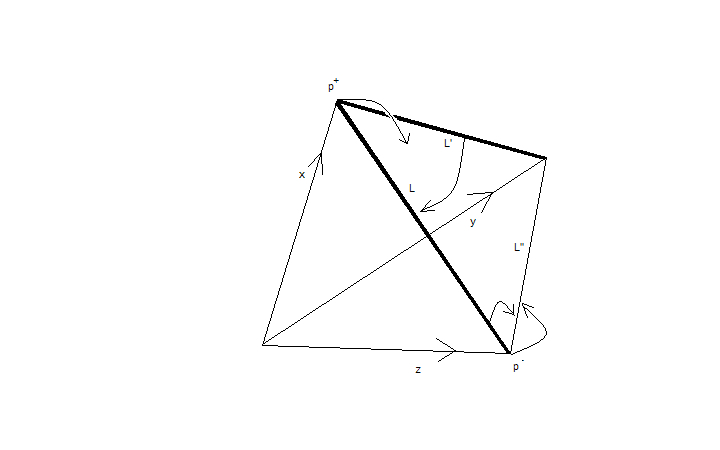}
\end{center}
\caption{Divisor at infinity of $\mathbb{P}^3(\mathbb{C})$. The indeterminacy set is represented by straight lines and the action by $f$ by curved arrows. }
\end{figure}

The main result of this section is the following.

\begin{proposition}
Let $X$ denote the blow-up of $\mathbb{P}^3$ having as smooth center the line $L=(Y=T=0)$, and let $f:X\rightarrow X$ denote the lift of $f$. Then $f:X\rightarrow X$ is algebraically stable. 
\end{proposition}

More precisely, we are going to show that the exceptional divisor $E$ of the blow-up $\pi:X\rightarrow\mathbb{P}^3$ along $L$ is mapped by $f$ onto regular points of $\widetilde{H}_{\infty}$, the strict transform of the hyperplane at infinity $H_{\infty}$, and that $\widetilde{H}_{\infty}$ is contracted by $f$ onto a regular superattractive fixed point. Let us still denote by $p^+$ and $p^-$ the points where $\widetilde{H}_{\infty}$ is mapped by $f^{-1}$ and $f$ respectively, and by $I_f$, $I_{f^{-1}}$ the indeterminacy locus for $f$, respectively $f^{-1}$. We denote also by $L'$, $L''$ the strict transform of $L',L''\subset \mathbb{P}^3$.

\begin{proof}
\noindent{\bf Image of $\widetilde{H}_{\infty}$ by $f$.} We first express $f$ in the chart $Z\neq 0$ at the input and output. 
The origin of this chart is ${p}^{-}$. We have
 \begin{align*}
 f(x,y,t)& :=f[x:y:1:t]=[yt:t:y+byt+ct+dxt:t^2]\\
 & \ =\left(\frac{yt}{y+byt+ct+dxt}, \frac{t}{y+byt+ct+dxt}, \frac{t^2}{y+byt+ct+dxt}\right).
 \end{align*}
 Recall that the blow-up variety $X$ is described by:
\begin{equation}\label{eclatement}
 \left\{ \left(x, y, t, [\xi_1, \xi_2] \right)\in \mathbf{C}^3\times \mathbf{P}^1, \xi_1 t=\xi_2 y \right\}.
\end{equation} 
The expression of $f$ in the chart $\xi_1\neq 0$ at input and output is then:
\begin{equation}\label{eq}  \begin{array}{ccc}
                                             &   f         &         \\
(\omega_1, \omega_2, \omega_3)                              &    \longmapsto    &  (\frac{\omega_2\omega_3}{A}, \frac{\omega_3}{A}, \omega_2\omega_3) \\
                                                            &                                            &  \\
                                                           \downarrow\pi &                                            &  \downarrow\pi\\
                                                            &                                            &  \\
 (\omega_1, \omega_2, \omega_2\omega_3) &    \longmapsto   & (\frac{\omega_2\omega_3}{A}, \frac{\omega_3}{A}, \frac{\omega_2\omega_3^2}{A})\\
                                                 
\end{array}\end{equation}

\noindent where $A=1+b\omega_2\omega_3+c\omega_3+d\omega_1\omega_3$. Note that the point $p^-$ at the origin of this chart is fixed, superattractive; we verify that $f(\widetilde{H}_{\infty})=p^-$. It remains to prove that $E$ is never mapped by $f^n$ onto a point of indeterminacy.
\smallskip

\noindent{\bf Image of $E$ by $f$. }The exceptional divisor $(\omega_2=0)$ is mapped by $f$ (except for the points of indeterminacy) on the strict transform of the line at infinity $L'':=(x=t=0)$.

Note for later use that there is no point of indeterminacy in $E\cap \widetilde{H}_{\infty}$ in the chart $(Z\neq 0, \xi_1\neq 0)$.

The expression of $f$ in the charts $\xi_2\neq 0$ at input and $\xi_1\neq 0$ at output is
\begin{equation}\label{eq1}  
\begin{array}{ccc}
                                             &   f         &         \\
(\omega_1, \omega_2, \omega_3)                              &    \longmapsto    &  (\frac{\omega_2\omega_3}{B},\frac{1}{B} , \omega_3) \\
                                                            &                                            &  \\
                                                           \downarrow\pi &                                            &  \downarrow\pi\\
                                                            &                                            &  \\
 (\omega_1, \omega_2\omega_3, \omega_3) &    \longmapsto   & (\frac{\omega_2\omega_3}{B}, \frac{1}{B}, \frac{\omega_3}{B})\\                                  
\end{array}
\end{equation}
where $B=\omega_2+b\omega_2\omega_3+c+d\omega_1$. Here we see again that the exceptional divisor $\omega_3=0$ is mapped by $f$ on the strict transform of the line at infinity $L''=(t=x=0)$. Therefore $E$ is contracted by $f^2$ onto $p^-$. \qed
\end{proof}
Let us summarize the action of $f$ at infinity:
$$\begin{array}{lcll}
\widetilde{H}_{\infty}\setminus I_{f} &\longrightarrow  &  p^-                       &     \\
E\setminus I_f                              &\longrightarrow    &  L''      \longrightarrow   &  p^-. \\

\end{array}$$
The sequence of indeterminacy sets $I_{f^n}\subset X$ of $f^{-n}$ (resp. $I_{f^{-n}}\subset X$ of $f^{-n}$) is stationary. More precisely, we have the following

\begin{proposition}\label{i} 
For all $n\geq 2$, $I_{f^n}= I_{f^2}:=I^+$ and $I_{f^{-n}}=I_{f^{-2}}:=I^-$. 
Each of the sets $I^+$ and $I^-$ consists of three curves.
Moreover, the set $I^+\cap I^-$ is finite with $$ \sharp(I^+\cap I^-)=\Bigg\{\begin{array}{cl}
                                      
 5                    & \text{ if }\left( b-c\right)^2-4d\neq 0, \\
 4                    & \text{ otherwise. }\\  
\end{array}$$
\end{proposition}

In Figure \ref{fig2}, $I^+$ is represented in the divisor at infinity $E\cup \widetilde{H}_{\infty}$ of $X$. The dynamics of the irreducible components of $I^+$ by $f$ is indicated by curved arrows.

\begin{figure}\label{fig2}
\begin{center}
\includegraphics[width=13cm, height=8cm]{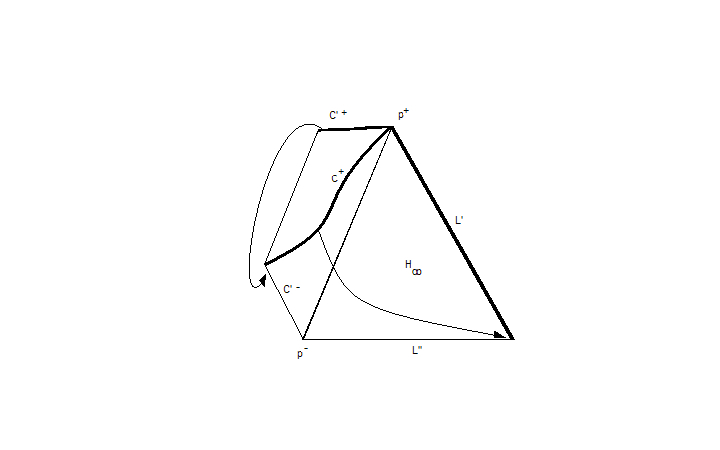}
\end{center}
\caption{Divisor at infinity of $X$. The indeterminacy set is represented by straight lines and the action by $f$ by curved arrows. }
\end{figure}

\begin{proof}

\noindent{\bf Points of indeterminacy in $E$ for $f$ and $f^{-1}$. } There are points of indeterminacy for $f$ in $E$. To see this, let us first consider the expression of $f$ in the charts $(z\neq 0)$ at input and $(y\neq 0)$ at output: we have  $  f(x,y,t)=\left(y, \frac{y+byt+ct+dxt}{t}, t\right)$. The expression of $f$ in the chart $(z\neq 0,\text{ }\xi_2\neq 0)$ at input and the chart $(y\neq 0)$ at output is:

\begin{equation}\label{cavalavie?}  \begin{array}{ccc}
                                             &   f         &         \\
(\omega_1, \omega_2, \omega_3)                              &    \longmapsto    &  (\omega_2\omega_3, \omega_2+b\omega_2\omega_3+c+d\omega_1, \omega_3) \\
                                                            &                                            &  \\
                                                           \downarrow\pi &                                            &  \downarrow\pi\\
                                                            &                                            &  \\
 (\omega_1, \omega_2\omega_3, \omega_3) &    \longmapsto   & (\omega_2\omega_3, \omega_2+b\omega_2\omega_3+c+d\omega_1, \omega_3)\\
                                     
\end{array}\end{equation}

\noindent Thus there is no point of indeterminacy in the chart $(Z\neq 0,\xi_2\neq 0)$. Moreover, we have seen that $E\cap \widetilde{H}_{\infty}$ does not contain any point of indeterminacy for $f$ in the chart $(Z\neq 0,\xi_1\neq 0)$. Now the local expression of $f$ in the charts $(X\neq 0)$ at input and $(Z\neq 0)$ at output is $$ f=\left(\frac{yt}{yz+byt+czt+dt}, \frac{zt}{yz+byt+czt+dt}, \frac{t^2}{yz+byt+czt+dt} \right).$$ The expression of the blow-up in the chart $\xi_2\neq 0$ at input and output is: 
\begin{equation}\label{eq.prout}  
\begin{array}{ccc}
                                             &   f         &         \\
(\omega_1, \omega_2, \omega_3)                              &    \longmapsto    &  (\frac{\omega_1\omega_3}{C}, \frac{\omega_2}{\omega_3}, \frac{\omega_3 }{C})\\
                                                            &                                            &  \\
                                                           \downarrow\pi &                                            &  \downarrow\pi\\
                                                            &                                            &  \\
 (\omega_1\omega_3, \omega_2, \omega_3) &    \longmapsto   & (\frac{\omega_1\omega_3}{C}, \frac{\omega_2}{C}, \frac{\omega_3}{C})\\                                               
\end{array}
\end{equation}
where $C:=\omega_1\omega_2+b\omega_1\omega_3+c\omega_2+d$. Therefore, there is an indeterminacy locus on $E$, that we call $\mathcal{C'}^+$, whose equation is $(\omega_2=0)$ in this chart. 
By symmetry, $\mathcal{C'}^-:=I_{f^{-1}}\cap E$ is given by $(\omega_2=0)$ in the chart $(z\neq 0$, $\xi_2\neq 0)$.
One sees by (\ref{eq.prout}) that $ f(I_f\cap E)\subset E$ and $ f(I_{f^{-1}}\cap E)\subset E$. Therefore 
$$
f(\mathcal{C'}^+)=\mathcal{C'}^-\text{ and }f^{-1}(\mathcal{C'}^-)=\mathcal{C'}^+.
$$
\smallskip

\noindent{\bf Points of indeterminacy in $E$ for $f^2$ and $f^{-2}$. }Points in $E\cap \left(I_{f^2}\setminus I_f\right)$ are mapped by $f$ on $[0,1,0,0]$. It follows from (\ref{cavalavie?}) that $\mathcal{C}^+:=E\cap \left(I_{f^2}\setminus I_f\right)$ can be expressed in the chart $(z\neq 0, \xi_2\neq 0)$ by
$
c+d\omega_1=0.
$
We compute in the same way the expression of $\mathcal{C}^-:=E\cap I_{f^{-2}}\setminus I_{f^{-1}}$ in the chart $(x\neq 0, \xi_2\neq 0)$, namely
$
\omega_2-\omega_1-b-c\omega_1\omega_3=0,
$ whereas in the chart $(x\neq 0, \xi_1\neq 0)$ it is given by the equation
$
\omega_2\omega_3-1-b\omega_3-c\omega_1\omega_3=0.
$
\smallskip

\noindent{\bf The set $I^+\cap I^-$.} In the chart $(z\neq 0, \xi_1\neq 0)$, the curve $\mathcal{C}^+$ in $E$ corresponds to the equation 
$1+c\omega_3+d\omega_1\omega_3=0$ and the curve $\mathcal{C}^-$ to $\omega_3-\omega_1-b\omega_1\omega_3=0$. 
A straightforward computation shows that $\mathcal{C}^+\cap \mathcal{C}^-$ has two points in $E$ which coincide when $\left( b-c\right)^2-4d=0$.
\smallskip

\noindent{\bf Sequences $I_{f^n}$ and $I_{f^{-n}}$ are stationary.} Since neither points of $E\setminus \{ \mathcal{C}^+, \mathcal{C}^-\} $ nor points of $\widetilde{H}_{\infty}\setminus L'$ are mapped by $f$ in $I_{f^2}$, we have $I_{f^2}=I_{f^3}=I_{f^n}$, $n \geq 2$.
Similarly $I_{f^{-n}}=I_{f^{-2}}$ for all $n \geq 2$. \qed
\end{proof}

\begin{remark}
One can verify that $f$ and $f^{-1}$ become algebraically stable after blowing-up the two points $p^+$ and $p^-$ where the hyperplane at infinity in $\mathbb{P}^3$ is contracted by $f$ and $f^{-1}$ respectively. Note however that this increases the number of charts to be considered, as well as the dimension of $H^{1,1}(X,\mathbf{R})$. This is one reason why we have used a slightly different compactification.
\end{remark}

\section{Canonical invariant classes}

The action of $f^*$ on the divisors at infinity is described by the following

\begin{lemma}
$f^*[\widetilde{H}_{\infty}]=[\widetilde{H}_{\infty}]+[E]$ and $f^*[E]=[\widetilde{H}_{\infty}]$.
\end{lemma}
\begin{proof}
We have $f(\widetilde{H}_{\infty}\setminus I_f)=p^-\in E$, and no other hypersurface at infinity is mapped in $E$. So 
$$
f^*[E]=a[\widetilde{H}_{\infty}], \text{ }a>0.
$$
Now equation (\ref{eq}) shows that $a=1$.

Since $\widetilde{H}_{\infty}$ and $E$ are both mapped into $\widetilde{H}_{\infty}$ by $f$ we have 
$$
f^*[\widetilde{H}_{\infty}]=b[\widetilde{H}_{\infty}]+c[E], \text{ }b,c>0.
$$
The coefficients $a=b$ and $c$ are the order of vanishing of $f$ along $\widetilde{H}_{\infty}$ and $[E]$ respectively. 
By equations (\ref{eq}) and (\ref{eq1}) we have $b=c=1$.\qed
\end{proof}

Thus the action of $f^*$ induced on $H^{1,1}(X,\mathbf{R})$ is represented by the matrix 
 $  \left(\begin{array}{cc}                               
1                     & 1 \\
1                     & 0 \\  
\end{array}\right)$ 
in the basis $(\{ \widetilde{H}_{\infty}\} ,\{  E\} )$. By symmetry, the action of $f^{-1*}$ induced on $H^{1,1}(X,\mathbf{R})$ is represented by the same matrix. The spectral radius of $f^*$ (resp. $f^{-1*})$) is equal to the golden mean $\sigma:=  \frac{1+\sqrt{5}}{2}$. The other eigenvalue of $f^*$ (resp. $f^{-1*}$) is $\frac{1-\sqrt{5}}{2}$. Recall that the {\it dynamical degrees} of $f$ are given by \begin{equation}\label{degredynamique} \lambda_i(f):=\lim_{n\rightarrow +\infty}\left( \int_{X} f^{n*}\left(\omega^i\right)\wedge \omega^{3-i}  \right)^{\frac{1}{n}}, i=1,2,3\end{equation} for a 
K\"ahler form $\omega$ on $X$ (see \cite{Gs}). A rational endomorphism is called {\it cohomologically hyperbolic} if one dynamical degree strictly dominates the others. We have thus shown the following (see also \cite{M}):

\begin{proposition}
The dynamical degrees of $f$ are 
$$
\lambda_1(f)=\lambda_2(f)=\lambda_1(f^{-1})= \frac{1+\sqrt{5}}{2}, \; \; \; \lambda_3(f)=\lambda_3(f^{-1})=1.
$$ 
Thus $f$ is not cohomologically hyperbolic.
\end{proposition}

There is a unique cohomology class $\alpha\in H^{1,1}(X,\mathbf{R})$ (up to a multiplicative constant) that is invariant by $ \frac{1}{\sigma}f^*$ (resp. $ \frac{1}{\sigma}f^{-1*}$). The class $\alpha$ is equal to $\{  E\}  +\sigma\{  \widetilde{H}_{\infty}\}  $.

We now study the positivity properties of the invariant class $\a$.
 
\begin{proposition}\label{21}
The invariant class $\alpha\in H^{1,1}(X, \mathbf{R})$ is K\"ahler.
\end{proposition}

\begin{proof}

{\bf{We start by showing that $\alpha$ is semi-positive.}} We equip the line bundle $\mathcal{O}_{\mathbb{P}^3}(H_{\infty})$ with a singular positive metric (\cite{D}) whose curvature form $\omega'$ is smooth outside the line $L=(Y=T=0)$ and which has 
(homogeneous) Lelong number equal to $1$ along this projective line. 
Note that $\omega'':=\pi^*(\omega')-[E]$ is a smooth semi-positive metric on the line bundle $\mathcal{O}_{X}(\widetilde{H}_{\infty})\simeq\pi^*\mathcal{O}_{\mathbb{P}^3}({H}_{\infty})\otimes\mathcal{O}_{X}(-E)$ 
which is therefore semi-positive. Thus the invariant class $\alpha=\{ \pi^*[H_{\infty}]+(\sigma-1)[\widetilde{H}_{\infty}]\} $ is semi-positive.

\smallskip

\noindent{\bf{The class $\alpha$ is  K\"ahler.}} Let $\omega_{FS}$ denote the Fubini-Study K\"ahler form on $\mathbb{P}^3$. It is cohomologous to $[H_{\infty}]$. It is classical (see e.g. \cite{GH}) that there exists a smooth closed form $\theta\in-\{ E\}$ such that $\pi^*\omega_{FS}+\varepsilon\theta>0$ for $\varepsilon>0$ sufficiently small. So 
$$
\alpha=(1+\varepsilon)\pi^*\{  H_{\infty}\}  +\varepsilon \{ \theta \} + (\sigma-1-\varepsilon)\{  \widetilde{H}_{\infty}\}  
$$
is  K\"ahler.\\
Here is a slight reformulation: let $\varphi:\mathbb{P}^3\rightarrow\P2$ be the projection on the hyperplane $(Z=0)\approx \mathbf{P}^2$ orthogonal to the line $L=(Y=T=0)$, $$\varphi([X,Y,Z,T])=[X,Y,0,T].$$
Then, after blowing-up $L$, the meromorphic application $\varphi$ induces a holomorphic map $\Phi: X\rightarrow \mathbf{P}^2$. We have 
$$
\begin{array}{rl} 0\leq \Phi^*\mathcal{O}_{\mathbf{P}^2}\left(Z=T=0\right) & \approx \pi^*\mathcal{O}_{\mathbf{P}^3}(T=0)\otimes
\mathcal{O}_{X}\left(-E \right)\\ 
   &  \\
  &   \approx \pi^*\mathcal{O}_{\mathbf{P}^3}(1)\otimes
\mathcal{O}_{X}\left(-E\right)\approx\mathcal{O}_X\left(\widetilde{H}_{\infty}\right)\end{array}
$$
where we have denoted $\approx$ the isomorphism between holomorphic line bundles. Therefore $ \{  \widetilde{H}_{\infty}\}  = c_1 \left( \mathcal{O}_X\left(\widetilde{H}_{\infty}\right)  \right)\geq 0$. So $ \{  \widetilde{H}_{\infty}\} $ can be represented by a smooth $(1,1)$-form that is positive on $E$. Since $\pi^* \omega_{FS} $ is smooth and positive except on $E$, we conclude that $\alpha = \pi^*\{  H_{\infty}\}  + (\sigma-1)\{  \widetilde{H}_{\infty}\}  $ is  K\"ahler.\qed
\end{proof}

\begin{remark}\label{prooout!}
We can describe explicitly in $\mathbb{C}^3$ the $(1,1)$-form $\omega'$ as follows. The function $\frac{1}{2}\log(|y|^2+|t|^2)$ defined in $\mathbb{C}^3=(x,y,t)$ is in the Lelong class $\mathcal{L}(\mathbf{C}^3)$, i.e. we have asymptotically $\frac{1}{2}\log(|y|^2+|t|^2)\leq\frac{1}{2}\log(1+|x|^2+|y|^2+|z|^2)+O(1)$ (\cite{GZ}). The function $g:=\frac{1}{2}\log(|y|^2+|t|^2)-\frac{1}{2}\log(1+|x|^2+|y|^2+|t|^2)$ thus extends as a $\omega_{FS}$-plurisubharmonic function on $\mathbb{P}^3$. 
This extension (that we call $g$ again) writes in homogeneous coordinates $g=\frac{1}{2}\log(|Y|^2+|T|^2)-\frac{1}{2}\log(|X|^2+|Y|^2+|Z|^2+|T|^2)$
and $\omega'=\omega_{FS}+dd^c g$. 
 \end{remark}

We now compute the volume of the class $\alpha$. The intermediate computations in the proof will also be useful in the sequel.

\begin{proposition}\label{big}The class $\alpha=\{  E\} +\sigma \{  \widetilde{H}_{\infty}\} $ has volume $$ \dis \mathrm{Vol}(\alpha)= \alpha ^3=1+\frac{3}{\sigma}>0.$$
\end{proposition}

\begin{proof}
A potential of $\omega\in \alpha$ in $\mathbb{C}^3$ is $\dis\frac{1}{2\sigma}\log(|y|^2+1)+\frac{1}{2}\log(|x|^2+|y|^2+|z|^2+1).
$
Set $\dis u:=\frac{1}{2\sigma}\log(|y|^2+1)$ and $\dis v:=\frac{1}{2}\log(|x|^2+|y|^2+|z|^2+1)$. Then
$$
\text{Vol}(\alpha) = \dis \alpha^3   =\int_{\mathbb{C}^3} \left( dd^c u+dd^c v\right)^3 
=\int_{\mathbb{C}^3}\left(dd^c v\right)^3+3(dd^c v)^2\wedge dd^c u
= 1+\frac{3}{\sigma}. 
$$
\flushright \qed
\end{proof}

\section{Invariant currents}

\subsection{The Green currents $T_f$ and $T_{f^{-1}}$}

Recall that $\omega_{FS}$ is the Fubini-Study form on $\mathbb{P}^3$ and that $\omega''$ is a smooth $(1,1)$-form on $X$ cohomologous to $\widetilde{H}_{\infty}$. It follows from the previous section that $\omega:=\pi^*\omega_{FS}+(\sigma-1)\omega''$ is a  K\"ahler form representing the invariant class $\alpha=\{  E\} +\sigma\{ \widetilde{H}_{\infty}\} $. Since $f^*\alpha=\sigma\alpha$, it follows from the $dd^c$-lemma that there exists an 
$\omega$-plurisubharmonic function $\phi$ such that 
$
 \frac{1}{\sigma}f^*\omega=\omega+dd^c \phi.
 $
 We normalize $\phi$ so that $ \sup_X \phi = 0$. 
 
 We refer the reader to \cite{GZ} for basic properties of $\omega$-plurisubharmonic functions.
 The sequence of positive currents 
 $$
  T_n:=\frac{1}{\sigma^n}f^{n*}\omega=\omega+dd^c \left( \sum_{i=0}^{n-1}\frac{1}{\sigma^i}\phi\circ f^i \right)
  $$
 converges towards a positive closed current $T_f$, since the decreasing sequence of $\omega$-plurisubharmonic functions  
 $\sum_{i=0}^{n-1}\frac{1}{\sigma^i}\phi\circ f^n$ converges in $L^1(X)$ 
 (it does not converge identically to $-\infty$ by an argument due to Sibony, see e.g. the proof of Theorem 2.1 in \cite{G}). 
 By construction, the current $T_f$ is invariant by $ \frac{1}{\sigma}f^*$.
 
Note that the "Green function", i.e. the potential of $T_f$ in $\mathbf{C}^3$ 
\begin{equation}\label{green}
G^+=\lim_{n\rightarrow+\infty}\frac{1}{\sigma^n}{f^n}^*\left(\frac{1}{2\sigma}\log(|y|^2+1)+\frac{1}{2}\log(|x|^2+|y|^2+|z|^2+1)\right)
\end{equation}
(See Remark \ref{prooout!}) 
is well defined and plurisubharmonic in $\mathbb{C}^3$. We construct in the same way the invariant current  for $f^{-1}$,
$$ 
T_{f^{-1}}:=\lim_{n\rightarrow +\infty}\frac{1}{\sigma^n}(f^{-n})^*\omega.
$$

Let $\mathcal{B}_{\infty}^+$ (resp. $\mathcal{B}_{\infty}^-$) denote the basin of attraction of the superattractive point $p^-$ (resp. $p^+$) at infinity.  
The following result is an easy adaptation (a special case) of Theorem 2.2 in \cite{GS2}:

\begin{proposition}\label{julia}
The potential $G^+\geq 0$ (resp. $G^-\geq 0$) of the current $T_f$ (resp. $T_{f^{-1}}$) is pluriharmonic in $\mathcal{B}_{\infty}^+$ (resp. $\mathcal{B}_{\infty}^-$). Moreover 
$$
\mathcal{B}_{\infty}^+=(G^+>0), \; \; \mathcal{B}_{\infty}^-=(G^->0) 
$$
and $G^+, G^-$ are continuous in $\mathbb{C}^3$.
We also have $\text{supp}(T_{f^{\pm}})\cap D_{\infty}=I^{\pm}$.
\end{proposition}


\subsection{The invariant current $T_f\wedge T_{f^{-1}}$ }

The $(2,2)$-current $T_f\wedge T_{f^{-1}}$ is well defined in $\mathbf{C}^3$ according to \cite{D1}.  

\begin{proposition}\label{cluster}
The support of the current $T_f\wedge T_{f^{-1}}$ clusters at infinity exactly along $I^+\cap I^-$, where it has positive Lelong numbers. 
\end{proposition}

\begin{proof}
We know by Proposition 6.3 in \cite{DF} that $T_f$ (resp. $T_{f^{-1}}$) has positive Lelong numbers at $I^+$ (resp. $I^-$). Thus $T_f\wedge T_{f^{-1}}$ has positive Lelong numbers at $I^+\cap I^-$ by Corollary 5.10 in \cite{D1}. So $ I^+\cap I^-\subset \text{supp}\left(T_f\wedge T_{f^{-1}}\right)$. 

Now take $ x\in D_{\infty}\setminus I^+\cap I^-$. 
By symmetry we can assume $x\notin I^+$. 
We know by Proposition \ref{julia} that $x\notin \text{supp}(T_f)$,
hence $x \notin \text{supp}(T_f \wedge T_{f^{-1}})$. Therefore $ I^+\cap I^- = \text{supp}\left(T_f\wedge T_{f^{-1}}\right)\cap D_{\infty}$.\qed
\end{proof}

The current $T_f\wedge T_{f^{-1}}$ is invariant by $  f^*$ in $\mathbf{C}^3$ since
\begin{equation}\label{bof}  
f^*(T_f\wedge T_{f^{-1}})=f^*T_f\wedge f^*T_{f^{-1}}=\sigma T_f\wedge \frac{1}{\sigma}T_{f^{-1}}= T_f\wedge T_{f^{-1}}.
\end{equation}
Note that this equality does not hold at infinity in $X$, because $  f^*\left(T_f\wedge T_{f^{-1}}\right)\geq T_f\wedge T_{f^{-1}}+a[\mathcal{C}^+]$, $a>0$, since $\nu(f^*S,z)\geq \nu(S,f(z))$. Of course $T_f\wedge T_{f^{-1}}$ is also invariant by $(f^{-1})^*$ in $\mathbb{C}^3$.


\smallskip

Set $J:=\text{supp}\left(T_f\wedge T_{f^{-1}}\right) $. Recall that $J\cap D_{\infty}$ consists in five points (Proposition \ref{i}). 
Using the same notations as in Proposition \ref{i}, we denote by:
\begin{itemize}
\item  $ A^+\in J\cap E$ the only point in $\mathcal{C'}^+$;\\
\item  $ A^-\in J\cap E$ the only point in $\mathcal{C'}^-$;\\
 \item $ B,B'\in J\cap E$ the two points in $ \mathcal{C}^+\cap\mathcal{C}^- $;\\
\item  $ C\in J\cap H_{\infty}$ the point such that $\pi(C)=[0,1,0,0]$.\\
\end{itemize}

We now analyse the action of $f_{|J}$ at infinity.

\begin{proposition}\label{prout}
Setting $f= f|_J$, we have
\begin{itemize}
\item $f(B)=f^{-1}(B)=f(B')=f^{-1}(B')=C$;
\item $f(C)=\{  A^- ,  B , B' \} $ and $f^{-1} (C)=\{  A^+ ,B ,B' \} $;
\item $f(A^+)= A^-$ and $f^{-1} (A^-)= A^+$;
\item $f(A^-)=C$ and $f^{-1}(A^+)=C$.
\end{itemize}
\end{proposition}

\begin{proof}
Since $f(\mathcal{C}^+)=C$, the regular points (for $f$ but not for $f^2$) $B,B', A^-\in \mathcal{C}^+$ are mapped into $C$ by $f$. We also obtain in the same way $f^{-1}(B)=f^{-1}(B')=f^{-1}(A)=C$. Since $A^+\in \mathcal{C'}^+$, we have $ f(A^+)\in \mathcal{C'}^-\cap J$, i.e. $f(A^+)=A^-$. A similar argument yields $f^{-1}(A^-)=A^+$.

\smallskip

Choose a sequence $ (p_n) \in J^{\N}$ that converges to $B$. Then the sequence $ z_n:=f(p_n)\in \text{supp}\left(T_f\wedge T_{f^{-1}}\right)$ 
converges to $C$. Therefore $B\in f^{-1} (C)$, since $ (z_n)\in \text{supp}\left(T_f\wedge T_{f^{-1}}\right)$. We prove in the same way that $B',A'\in f^{-1} (C)$, and that there are no other possibilities. Thus $f^{-1} (C)=\{  A^+,  B ,  B' \} $, and similarly we find $f(C)=\{  A^- , B ,  B' \} $.\qed
\end{proof}



For two unbounded sets $S$, $S'\in\mathbb{C}^3$, we write $S\simeq S'$ if there exists $R\in\mathbb{R}^+$ such that $S\cap \{\parallel z\parallel > R\} = S'\cap \{\parallel z\parallel > R\}$.  This defines an equivalence relation, whose classes will be called {\bf{germs at infinity}}. Moreover, if we have a partition $S\cap \{\parallel z\parallel > R\} = S_1\sqcup S_2$ for some $R>0$, we will write $S \simeq S_1 \sqcup S_2$. 

Let $J_{A^+}$, $J_{A^-}$, $J_{C}$, $J_B$  be the germs at infinity of the intersections of $J$ with disjoint neighborhoods of $A^+$, $A^-$, $C$, $\{B, B'\}$, respectively.
 It follows for instance from the previous proposition that
 \begin{equation}
 f(J_{A^+})\simeq J_{A^-}\text{ and }  f^{-1} (J_{A^-})\simeq J_{A^+}.
 \end{equation}
  Similarly, Proposition \ref{prout} yields
 \begin{equation}
 \label{two}J_C\simeq f^{-1} (J_{A^+})\sqcup f^{-1} (J_B)\simeq f(J_{A^-})\sqcup f(J_B).
 \end{equation}
 
 We can now describe the dynamics on $J$ near infinity. Given two sets $A,B\subset X$, we write $A\rightarrow B$ for $f(A)\simeq B$. 
 
 \begin{proposition}\label{pro.partition}
There exists a partition of $J|_{\mathbb{C}^3}$ near infinity into nonempty sets $$\displaystyle J_{\mathbb{C}^3}\simeq J'_{A^+}\sqcup J'_{A^-}\sqcup J'_{C}\sqcup J'_B\sqcup J''_{A^+}\sqcup J''_{A^-}\sqcup J''_{C}\sqcup J''_B,$$where $J'_K$, $J''_K\subset J_K$ for $K\in \{A^+, A^-, C, {B,B'}\}$, such that the following hold
$$
\begin{array}{ccccccccccccc} J_{A^+}' &\rightarrow & J_{A^-}' & \rightarrow & J_{C}' &\rightarrow &  J_{A^+}', & & & & & &\\ 
 &  & &   &  & &  & & & & & & \\  J_{A^+}''  &\rightarrow & J_{A^-}'' & \rightarrow & J_{C}'' &\rightarrow &  J_{B}'' &\rightarrow & J_{C}'' & \rightarrow & J_{A^+}'' .\end{array}
 $$

Moreover, setting $J':= J'_{A^+}\sqcup J'_{A^-}\sqcup J'_{C}\sqcup J'_B$ and $J'':= J''_{A^+}\sqcup J''_{A^-}\sqcup J''_{C}\sqcup J''_B$, the restriction of $f$ to $J'$ or $J''$ extends continuously to infinity at $A^+$, $A^-$, $C$.
\end{proposition}



Thus the points in $J$ near infinity are "weakly periodic with period $3$ or $5$", in the sense that each of them is mapped near itself by $f^3$ or $f^5$. 
We summarize the above proposition in Figure 3.

\begin{figure}
\begin{center}
\includegraphics[width=12cm, height=8cm]{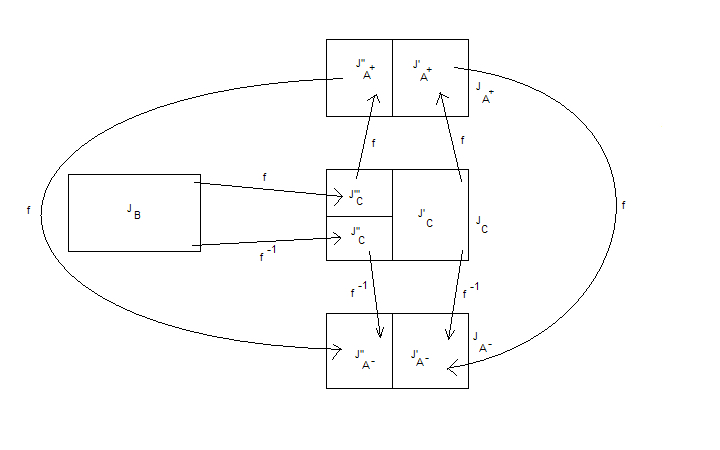}
\end{center}
\caption{Dynamics near infinity on the support of $T_f\wedge T_{f^{-1}}$.}
\end{figure}

\begin{proof}
We set 
$
J'''_{C}:=f(J_B),\text{ }J''_C:=f^{-1}(J_B),\text{ }J'_{C}=J_C\setminus (J''_C\cup J'''_C).
$

We will first prove that these three sets form a partition of $J_C$.
Note that Proposition \ref{prout} implies $ f(J_C\setminus J''_C)\subset J_{A^+}$. 
 It follows from (\ref{two}) that
\begin{equation}\label{*}
f(J_C\setminus J''_C)\simeq J_{A^+}.
\end{equation}
Similarly,
$
f^{-1}(J_C\setminus J'''_C)\simeq J_{A^-}.
$

Now, by Proposition \ref{prout} we have $ f^2(J_{A^-})\subset J_B\sqcup f(J_C\setminus J''_C)$. 
Using (\ref{*}) and the fact that $f_{|J}:J \rightarrow J$ is surjective, we obtain 
$
 f^2(J_{A^-})\simeq J_B\sqcup J_{A^+}.
 $
 
Rewriting this equation as $ J_{A^-}\simeq f^{-1} (J''_C)\sqcup f^{-2}(J_{A^+})$ yields 
$$
 f^{-1} (J''_C)\subset J_{A^-}.
 $$
This equation, combined with $f (J'''_C)\subset J_{A^+}$ obtained in the same way, shows that $J'''_C\cap J''_C=\emptyset$. So we have proved that $J_C\simeq J'_C\sqcup J''_C\sqcup J'''_C$.
\smallskip

Set $J''_{A^+}:=f(J'''_C)$, $J'_{A^+}=f(J'_C)$, $J''_{A^-}:=f^{-1} (J'''_C)$, $J'_{A^-}=f^{-1} (J'_C)$. It follows from the previous paragraph that $J_{A^+}\simeq J'_{A^+}\sqcup J''_{A^+}$ and $J_{A^-}\simeq J'_{A^-}\sqcup J''_{A^-}$. Indeed, the sets $J'_{A^+}$ and $J''_{A^+}$ are disjoint since $J'_C$ and $J'''_C$ are, and the surjectivity of $f_{|J}$ shows that $J'_{A^+}\sqcup J''_{A^+}$ is a partition of $J_{A^+}$.

\smallskip

Let us now prove that $f^3(J_B)\simeq f^{-2}(J_B)$ (this means informally that the previous diagram is "commutative" in a weak sense). We have \begin{equation}\label{sqcup}
\begin{array}{rll} J_{A^+} & = & \left(f^3(J_{A^+})\cap J_{A^+} \right)\sqcup \left( f^3\left(J_C''\right)\cap J_{A^+}\right)\\
   &   &   \\
   &   = &   \left(f^3(J_{A^+})\cap J_{A^+} \right)\sqcup\left(J_{A^+}''\cap J_{A^+}\right)\end{array}.
   \end{equation}
Indeed, a point $x\in J_{A^+}$ can only come from $J_{A^+}$ or $J_c''$ by $f^3$, according to Proposition \ref{prout}. It follows that 
   $
    f(J_{A^+}'')\subset J_{A^-}''.
    $
    
Similarly, one gets $ f^{-1} (J_{A^-}'')\supset J_{A^+}''$ and $ f(J_{A^+}'')\subset J_{A^-}''$.
\smallskip

\noindent{\bf The set $J''$ is nonempty, and the set $J'$ is nonempty for $|b|<1$ and  $|c|$ sufficiently small.} By Proposition \ref{cluster}, the set $J=J'\sqcup J''$ reaches infinity at $B$. Since we have seen that $J'$ does not reach infinity at $B$ but $J''$ does, the set $J''_B$ is nonempty. Therefore, its images by $f$ and $f^2$ are nonempty also. Thus $J''$ is nonempty.

For proving that $J'$ is nonempty, it is sufficient to prove that there exists a sufficiently small neighborhood of $A^-$ in $X$ that is mapped by $f^2$ into a neighborhood of $A^+$. It will be shown in Remark \ref{epsilon} below that $J_{A^-}\cap  \{(x,y,z)\in \mathbb{C}^3, |x|<\varepsilon, \text{ }|y|<\varepsilon^2 \}\neq \emptyset$. 
 If $|c|<\varepsilon^2$ and $|b|<1$, it is straightforward that $f(x_0,y_0,z_0)\in V_{A^+}$.\qed

\end{proof}



\begin{corollary}\label{f15}
$f^{15}|_{J}:J\rightarrow J$ extends continuously at infinity. Moreover, the points at infinity are fixed by $f^{30}|_{J}:J\rightarrow J$.
\end{corollary}

\begin{proof}
Proposition \ref{pro.partition} shows that $f^{15}|_{J}:J\rightarrow J$ extends continuously at infinity, except perhaps at $B$ and $B'$. Let $J_B\simeq J_{B}'\sqcup J_{B}''$ denote the partition of $J_B$ into the two branches that contain the two points $B$ and $B'$. We can prove by the same type of arguments as those used in the proof of the previous proposition that $J_{C}''\simeq f^{-1}(J_B')\sqcup f^{-1}(J_{B}'')$ and $J_{C}'''\simeq f(J_B')\sqcup f(J_{B}'')$, and that $f^{5}|_{J_B'}$ and $f^{5}|_{J_B''}$ extend continuously at infinity.
\smallskip

It follows from Proposition \ref{pro.partition} that $f^{15}|_{J}:J\rightarrow J$ fixes the points $A^+,A^-,C$, and eventually permutes the points $B$ and $B'$. Thus all these points are fixed by $f^{30}|_{J}:J\rightarrow J$.\qed
\end{proof}


\begin{remark}\label{serviraplutar}
We deduce from Proposition \ref{julia} that $T_f\wedge T_f=T_{f^{-1}} \wedge T_{f^{-1}}=0$ in $\mathbb{C}^3$, since 
$T_f\wedge T_f=dd^c(G^+ T_f)$ and $G^+=0$ on $\text{supp}(T_f)\cap \mathbb{C}^3$. 
Thus neither $T_f\wedge T_f$ nor $T_{f^{-1}}\wedge T_{f^{-1}}$ are of any interest. \end{remark}

\section{A canonical invariant measure}

We now go back to dynamics. We study the same family as before, quadratic polynomial automorphims $f:\C^3\rightarrow \C^3$ which are not cohomologically hyperbolic.
We would like to construct a canonical invariant probability measure for them. Recall that
$$
f(x,y,z)=(y,z,yz+by+cz+dx+e), \; \;b,c,d\in\mathbb{C}, |d|>1
$$
\noindent induces an algebraically stable automorphism of $X$, the blow-up of $\mathbb{P}^3$ along a suitable projective line, such that $f\moins$ is also algebraically stable with
$$
\dis\lambda_1(f)=\lambda_1(f\moins)=\frac{1+\sqrt{5}}{2}>\lambda_3(f)=1.
$$
\noindent We have constructed the canonical invariant currents $T_f,T_{f\moins}$ and $T:=\frac{\sigma}{3+\sigma}T_f\wedge T_{f\moins}$. In this section we construct a continuous function $\phi_{\infty}$ defined on 
the support of the current $T$ with suitable invariance properties and such that it is $T$-plurisubharmonic. (Recall that a $T$-plurisubharmonic function, associed with a closed positive current $T$, is a function $u:supp(T)\rightarrow \mathbb{R}\cup\{-\infty\}$, $u\in L^1(|T|)$, that is the limit on $supp(T)$ of a sequence of plurisubharmonic functions defined on some neigborhood of $supp(T)$, decreasing on $supp(T)$.)

We then study the invariant measure $\mu=\dis\left(T_f\wedge T_{f\moins}\wedge  dd^c\phi_{\infty}\right){|_{\mathbb{C}^3}}$. We prove that its support is compact and included in the set of points with bounded forward orbits.


We suppose in this section that $|d|>1$ in Equation (\ref{maegawa}). We set 
$$ 
T:=\frac{\sigma}{3+\sigma}T_f\wedge T_{f^{-1}}.
$$ 
Recall that $J=\text{supp}(T)$. We can decompose $J=K\sqcup W$, where:
 
$$
\begin{array}{rl}
 K        &     \dis:=\{\text{the points of }J\text{ with a bounded forward orbit}\},\\
          &                                                                         \\
W         &                              :=J\setminus K.  \\
\end{array}
$$
Note that $K$ is compact. Indeed, Lemma \ref{below} below implies that $K$ is bounded, and the same lemma implies that $W$ is open in $\mathbb{C}^3$ as the basin of attraction of infinity on $J$. The set $W$ can be decomposed again into two "crossed" partitions:
\begin{lemma}\label{W}
We can decompose $W=W'\sqcup W''$ in $\mathbb{C}^3$ with
 $$W':=\{z\in W: f^{3n}(z)\text{ converges towards a point at infinity when }n\rightarrow +\infty\}$$
 $$W'':=\{z\in W: f^{5n}(z)\text{ converges towards a point at infinity when }n\rightarrow +\infty\}.$$
 $W'$ and $W''$ are unbounded sets invariant by $f^3$ and $f^5$ respectively.
 \smallskip
 
\noindent We can also decompose 
 $$
 W:=W_{A^+}\sqcup W_{A^-}\sqcup W_{C}\sqcup W_B\sqcup W_{B'},
 $$
where $W_{A^+},W_{A^-},W_{C},W_B,W_{B'}$ are the connected components of $W$ in $X$ that reach the points $A^+$, $A^-$, $C$, $B$, $B'$ respectively at infinity. These five sets are invariant by $f^{15}$.
 \end{lemma}
 
\begin{proof}
A glance at Proposition \ref{pro.partition} shows that $W=W'\cup W''$, with the notations of the statement. It remains to prove that $W'\cap W''=\emptyset$. The set $W'\cap W''$ is stable by $f^3$ and $f^5$, therefore it is periodic with period $1$, i.e. it is stable by $f$. This set $W'\cap W''$ is not bounded, otherwise $W'\cap W''\subset K$, which is not possible by definition of $W$ and $K$. Thus $W'\cap W''$ reaches infinity. Then Proposition \ref{pro.partition} prevents $W'\cap W''$ from being fixed under the action of $f$, a contradiction.\\
Now consider the decomposition $W=W_{A^+} \cup W_{A^-}\cup W_{C}\cup W_B\cup W_{B'}$ given by Proposition \ref{pro.partition}. If $x\in W_{A^+}$, then $f^{15n}(x)\rightarrow A^+$ when $n\rightarrow +\infty$. If $W_{A^+}\cap W_C\neq \emptyset $, then we would also have $\dis\lim_{n\rightarrow +\infty} f^{15n}(x) =C $, which is a contradiction. We prove in the same way that the previous decomposition of $W$ is a genuine partition.\qed
\end{proof}

With the notations of Proposition \ref{pro.partition}, we have $J'_{A^+}, J'_{A^-}, J'_C\subset W'''$ and $J''_{A^+}, J''_{A^-}, J''_C\subset W''$. Our goal is to construct a function $\phi_{\infty}:J\rightarrow \mathbb{R}$ with suitable invariant properties, such that the current $T\wedge dd^c \phi_{\infty}$ makes sense and yields an invariant positive measure.

\subsection{Growth control and invariant measure}

We will now construct a $T$-plurisubharmonic function $\phi_{\infty}$ on $J$ with invariant properties. This will allow us to construct an invariant probability measure.
We set $|T|:=T\wedge \omega$
for a  K\"ahler form $\omega$ and denote by $\parallel \cdot\parallel$ the max norm, and\linebreak
$\parallel \cdot \parallel_+:=\max (\parallel \cdot \parallel , 1)$.

\smallskip

We summarize the construction of $\phi_{\infty}$ and of the invariant measure in the following theorem:

\begin{theorem}\label{measure}
Suppose $|d|>1$. There exists a $T$-plurisubharmonic function $\phi_{\infty}:J\rightarrow \mathbb{R}$ which satisfies the following properties:
\begin{enumerate}
\item The function $\phi_{\infty}$ satisfies $\phi_{\infty}\circ f^{3}=\phi_{\infty}+\log |d|$ on $W'$;
\item $\phi_{\infty}=\log^+\parallel z\parallel+o(1)$ as $\parallel z \parallel \rightarrow +\infty$;
\item $\left(\phi_{\infty}=-\infty\right)=K$ and $|T|(K)=0$;
\item $e^{\phi_{\infty}}$ is continuous on $J$.
\end{enumerate}
The measure $\mu:=T\wedge dd^c\left(\phi_{\infty}\right)|_{\mathbb{C}^3}$ is a well-defined $f^{3}$-invariant positive measure with compact support, such that
$
\text{supp} (\mu)\subset(\phi_{\infty}=-\infty).
$
\end{theorem}

\begin{proof}
Let $V_{A^+}$ be a small neighborhood of the point of indeterminacy $A^+$.
\smallskip


\noindent{\bf Construction of $\phi_{\infty}$.} It is sufficient to construct the function $\phi_{\infty}$ on $J_{A^+}:= K\sqcup W_{A^+}$ and to check the stated properties on this set. We first construct $\phi_{\infty}:J_{A^+}\rightarrow [-\infty,+\infty]$ as a limit of a decreasing sequence of functions. Let 
$$ 
\psi_n(p):=\log^+ \left(\parallel f^{3n}(p)\parallel\right)-n\log |d|.
$$
Then we define a decreasing sequence on $J_{A^+}$ by $\dis\psi'_n:={\sup_{i\geq n}}^*\psi_i$, and its limit
$$
\dis\phi_{\infty}:={\limsup_{n\rightarrow +\infty}}^*\psi_n,
$$
where we have denoted by ${\sup}^*$ the usc-regularization of the sup. Note now that the limit $\dis\phi_{\infty}:=\lim_{n\rightarrow +\infty}\psi'_n$ defined on $W_{A^+}\cap V_{A^+}$ satisfies $\phi_{\infty}\circ f^{3}=\phi_{\infty}+\log |d|$. Indeed,
\begin{align}\label{boum!} 
\phi_{\infty}\circ f^{3} &=\lim_{n\rightarrow +\infty}{\sup_{i\geq n}}^* \left(\psi_i\circ f^{3} \right) =\lim_{n\rightarrow +\infty}{\sup_{i\geq n}}^* \left(\psi_{i+1}\right)+\log |d| \nonumber \\
& =\lim_{n\rightarrow +\infty}\left(\psi'_{n+1}\right)+\log |d|=\phi_{\infty}+\log |d|.
\end{align}

\smallskip

\noindent{\bf The sequence $(\psi_n)$ is locally uniformly bounded on $W_{A^+}\cap V_{A^+}$}.

By Lemma \ref{below}, it suffices to verify that
\begin{equation}\label{eq.sigma}
\sigma_{n,\varepsilon}:= \sum_{i=1}^n \log \left(|d|+\frac{\varepsilon}{(|d|-\varepsilon)^{i-1}}  \right)
-\sum_{i=1}^n \log \left(|d|-\frac{\varepsilon}{(|d|-\varepsilon)^{i-1}}  \right)
\end{equation}
does not blow up. Now
$$
\sigma_{n,\varepsilon}= \sum_{i=1}^n\log \left(1+\frac{\varepsilon}{|d|(|d|-\varepsilon)^{i-1}}  \right)-\sum_{i=1}^n\log \left(1-\frac{\varepsilon}{|d|(|d|-\varepsilon)^{i-1}}  \right),
$$ 
hence $(\sigma_{n,\varepsilon})$ is a difference of uniformly convergent series, since
$$  \log \left(    1\pm  \frac{\varepsilon}{|d|(|d|-\varepsilon)^{i-1}  }\right)\thicksim \pm\frac{ \varepsilon}{|d|(|d|-\varepsilon )^{i-1}}.$$
Thus $(\psi_n)$ is locally uniformly bounded near infinity. Remark that, up to shrinking $V_{A^+}$, we can take $\varepsilon$ in the previous expression as small as we want: thus $\forall n\geq 0, \psi'_{n}(z)-\log^+\parallel z\parallel=o(1)$ as $\parallel z \parallel \rightarrow +\infty$.
\smallskip

\noindent{\bf We have $\phi_{\infty}-\log^+\parallel .\parallel=o(1)$ near infinity and the sequence $(\psi'_n)$ is locally uniformly bounded on $J$ near infinity.} Fix a neighborhood $V_{A^+}$ of $A^+$ and $\varepsilon>0$ such that the formula given by Lemma \ref{below} holds in $V_{A^+}\cap J$. Equation (\ref{eq.sigma}) implies that for all $n\geq 0$ we have on $V_{A^+}\cap J$ $$\dis\log^+\parallel z \parallel+c_{\varepsilon}\leq \psi_n(z)\leq \log^+\parallel z \parallel+c_{\varepsilon}$$ for some constant $c_{\varepsilon}>0$. We can take $\dis c_{\varepsilon}=\lim_{n\rightarrow +\infty}\sigma_{n,\varepsilon}<+\infty$, where $\sigma_{n,\varepsilon}$ is given by Equation (\ref{eq.sigma}). Then $\dis\lim_{\varepsilon\rightarrow 0}c_{\varepsilon}=0$. We also have on $V_{A^+}\cap J$, for all $n\geq 0$, $$\dis\log^+\parallel z \parallel+c_{\varepsilon}\leq \psi'_n(z)\leq \log^+\parallel z \parallel+c_{\varepsilon}.$$
Letting $n\rightarrow +\infty$, we obtain $$\dis\log^+\parallel z \parallel+c_{\varepsilon}\leq \phi_{\infty}(z)\leq \log^+\parallel z \parallel+c_{\varepsilon}$$on $V_{A^+}\cap J$. We conclude by shrinking $V_{A^+}$ and letting $\varepsilon\rightarrow 0$. 

Thus $\phi_{\infty}-\log^+\parallel .\parallel=o(1)$ near infinity. In particular, the sequence $(\psi'_n)$ is locally uniformly bounded on $J\cap V_{A^+}$.

\smallskip

\noindent{\bf The function $\phi_{\infty}$ is locally bounded on $W_{A^+}$.} Indeed, $\phi_{\infty}$ is locally bounded near infinity since $\phi_{\infty}-\log^+\parallel .\parallel=o(1)$ near infinity. Now let $x\in W$ and $N\geq 0$ such that $f^{3N}(x)\in V_{A^+}$. Thus $\phi_{\infty}(x)=\phi_{\infty}\circ f^{3N}-N\log|d|$. Since $\phi_{\infty}-\log^+\parallel \cdot \parallel =o(1)$ near infinity, we deduce that $-\infty<\phi_{\infty}(x)<+\infty$ on $W$. Obviously, $\phi_{\infty}|_K=-\infty$. 
\smallskip

\noindent{\bf{The function $\phi_{\infty}$ is continuous on $W_{A^+}$.}} Choose $\varepsilon$ and $V_{A}^+$ sufficiently small such that 
$\left|\phi_{\infty}-\log^+\parallel . \parallel \right|<\varepsilon$ in $W_{A^+}\cap V_{A}^+$. Take $x,y\in W$. There exists $N\geq 0$ such that $x,y\in f^{-N}(W_{A^+}\cap V_{A^+})$. Then 
$$
\begin{array}{rcl} \left|\phi_{\infty}(x)-\phi_{\infty}(y)\right| & = & \left|\phi_{\infty}\circ f^{3N}(x)-\phi_{\infty}\circ f^{3N}(y)\right|\\
  &  &  \\
   &  \leq  & \left|\log^+\parallel f^{3N}(x)\parallel -\log^+\parallel f^{3N}(y) \parallel\right|+2\varepsilon.\end{array}
$$
If $x\rightarrow y$ in $f^{-N}(W'\cap V_{A^+})$, we obtain 
$$\lim_{x\rightarrow y}|\phi_{\infty}(x)-\phi_{\infty}(y)|\leq 2\varepsilon.$$
The conclusion follows by letting $\e \rightarrow 0$.
\smallskip

\noindent{\bf{The function $e^{\phi_{\infty}}$ is continuous on $K\sqcup W_{A^+}$.}} It is sufficient to prove that $\phi_{\infty}(x)\rightarrow -\infty$ when $x\rightarrow \partial K$. Let, for $n \in \mathbb{N}$,
$$
V_n:=f^{-3n}(V_{A^+}\cap W)\setminus f^{-3(n-1)}(V_{A^+}\cap W)
$$
be disjoint bounded sets on $W$. Then, by Lemma \ref{below},
$
W=\cup_{k\in \mathbb{N}}V_k.
$
Fix $k_0>1$ an integer. There exists a real number $M$ such that $\phi_{\infty}<M$ on $V_{k_0}$. Take $k \geq k_0$. For all $x\in V_k$, there exists $N\geq 0$ such that $f^{3N}(x)\in V_{k_0}$. Then
$$
\phi_{\infty}(x)=\phi_{\infty}\circ f^{3N}(x)-N\log |d|<M-N\log |d|.
$$
Since $|d|>1$, the conclusion follows by noting that $N\rightarrow +\infty$ when $x\rightarrow \partial K$.

\smallskip


\noindent {\bf The measure $\mu:=T\wedge  dd^c\phi_{\infty}$ is well defined in $\mathbb{C}^3$}. First note that for all $k\in\mathbb{R}$, the function $\phi_{\infty,k}:=\max(\phi_{\infty},k):W' \rightarrow [-\infty,+\infty[$ can be extended by the value $k$ into a $T$-plurisubharmonic function on the whole $J$. For all $k\in\mathbb{R}$, the current $T\wedge dd^c\phi_{\infty,k}$ is defined by its action in $\mathbb{C}^3$ on a test function $\xi$:
$$
 \displaystyle<T_f\wedge T_{f^{-1}}\wedge dd^c\phi_{\infty,k}, \xi >:=<\phi_{\infty,k} T_f\wedge T_{f^{-1}}, dd^c \xi >.
$$
Note that its support is included in $W'$. We want to define the current $T\wedge dd^c\phi_{\infty}$ by its action in $\mathbb{C}^3$ on a test function $\xi$:
$$
 \displaystyle<T_f\wedge T_{f^{-1}}\wedge dd^c\phi_{\infty}, \xi >:=\lim_{k\rightarrow -\infty}<\phi_{\infty,k} T_f\wedge T_{f^{-1}}, dd^c \xi >.
$$

Since the support of $T\wedge dd^c\phi_{\infty,k}$ is included in $W'$, it suffices to check that $\phi_{\infty}\in L^1_{loc}(|T|)$. Note that this implies $ |T|\left(\phi_{\infty}=-\infty\right)=0$. Fix $R>0$. Let $\xi$ be a test function such that $dd^c\xi>0$ on the closed ball $\overline{B}(0,R)$. Since
$$
\lim_{x\rightarrow\infty, x\in J}\phi_{\infty}(x)=+\infty,
$$
we can suppose that $\psi'_{n}\leq 0$ on $\text{supp}(\xi)$ for some $N>0$ and $n\geq N$. Then for all $n\geq N$
$$
\dis -\int_{\mathbb{C}^3}|\psi'_{n}|T\wedge dd^c\xi=\int_{\mathbb{C}^3}\xi T\wedge dd^c \psi'_{n}\geq c\int_{\mathbb{C}^3}T\wedge dd^c\psi'_n
$$
for some $c>0$. We can show that this latter integral is bounded from below by a constant independant of $n$.
\smallskip

\noindent {\bf The support of $\mu$ is compact in $\mathbb{C}^3$.} 
To simplify notations we write $f$ for $f^{3}$. Note that $ f\left(W\cap V_{A^+}\right)\subset\subset W\cap V_{A^+}$ by Lemma $\ref{below}$. Since $f$ is an automorphism, we have for all $n\geq 0$ 
$$ \int_{W\cap V_{A^+}}\mu=\int_{W\cap V_{A^+}}\left(f^{n}\right)_*\mu=\int_{f^n\left(W\cap V_{A^+}\right)}\mu.$$
Letting $n\rightarrow +\infty$, we see that the previous integral is equal to $\mu(A^+)=0$.
\noindent For every compact $Q\subset W$, there exists $N\geq 0$ such that $Q\subset f^{-N}(W\cap V_{A^+})$. Thus $ \mu|_Q=f^{N*}\left(f^N_*\mu|_Q\right)=f^{N*}0=0$. Therefore $\text{supp}(\mu)\subset \overline{K}$. We have remarked at the beginning of this section that $K$ is closed. Thus $\text{supp}(\mu)\subset K$.
\smallskip



\noindent{\bf The total mass of $\mu$ is not vanishing}.  Denote $\infty:=J\cap E$, where $E$ is the exceptional divisor at infinity. 

\noindent Now let $A\in \mathbb{R}^-\setminus \{0\}$. There exists $n_0\in \mathbb{N}$ such that $\phi_{n}\leq A$ on $\{\phi_{\infty}<2A\}\subset W'$ for $n\geq n_0$. Let then $\delta>0$ sufficiently small for that $\phi_n< A$ on $V_{\delta}\cap W''$, where  $V_{\delta}$ is the $\delta$-neighborhood  of $W'$. Up to diminishing $\delta$, we can also suppose that $U_{\delta}:=V_{\delta}\cap \{\parallel z\parallel < 2\}$ does not intersect $W''$. Recall that
$\displaystyle \lim_{x\rightarrow\infty, x\in W'}\phi_{\infty}(x)-\log^+\parallel x \parallel=0.$ Then the Comparison Theorem (Corollary 2.7 in \cite{DE}) implies that $$\displaystyle  T\wedge dd^c \max(\phi_n, A)(U_{\delta})\geq \int_{U_{\delta}}T\wedge dd^c \log^+\parallel z\parallel^{\frac{1}{2}}>0$$
for all $n\geq n_0$. This latter value does not depend neither on $\delta$, $A$ nor $n$. Thus we can diminish $A$ towards $-\infty$ up to diminishing $\delta$ towards $0$, in order to obtain $$T\wedge dd^c \phi_{\infty} (U_{\delta})=\displaystyle \lim_{A \rightarrow -\infty} \lim_{\delta\rightarrow 0} T\wedge dd^c \max(\phi_n, A)(U_{\delta})\geq \int_{U_{\delta}}T\wedge dd^c log^+\parallel z\parallel^{\frac{1}{2}}>0.$$


 
 
\flushright \qed
\end{proof}

We now come to the main technical estimate used in the proof above. Recall that $J:=\text{supp}\left(T_f\wedge T_{f^-}\right)$ decomposes into three sets $J=K\sqcup W'\sqcup W''$ invariant by $f^{15}$, where $K\subset \mathbb{C}^3$ is compact and $W',W''$ are sets that reach infinity in four or five points and are periodic under $f$ with period $3$ and $5$ respectively.

\begin{lemma}\label{below}
 Suppose $|d|>1$. For all $\varepsilon>0$, there exists a neighborhood $V_{\varepsilon}$ of the divisor at infinity such that for all $p\in W'\cap V_{\varepsilon}$ and for all $n\geq 1$, 
 $$
  \parallel p \parallel_+\prod_{i=1}^n \left(|d|-\frac{\varepsilon}{(|d|-\varepsilon)^{i-1}}\right)\leq \parallel f^{3n}(p)\parallel_+
  \leq \parallel p \parallel_+\prod_{i=1}^n \left(|d|+\frac{\varepsilon}{(|d|-\varepsilon)^{i-1}}\right).
  $$
  
\noindent In particular, $W' \cap E$ is attracting for $f^{3}$ on $W'$, in the sense that 
 $$ 
 f^{3}\left(W'\cap V_{\varepsilon}\right)\subset\subset\left(W'\cap V_{\varepsilon}\right).
 $$
\end{lemma}

\begin{proof}

For $\varepsilon \in ]0, d[$,  denote $V_{A^+, \varepsilon} := \{|yz+by+cz| < \varepsilon|dx|\}$, $V_{A^-, \varepsilon} := \{|xy+bx+cy| < \varepsilon|dz|\}$, and
 $V_{C, \frac{\varepsilon}{d}} := \{|xz+by+cz| < \varepsilon|dy|\}$.\\
 
\noindent {\bf{The sets $V_{A^+, \varepsilon}$ and $V_{A^-, \varepsilon}$ are neighborhoods of $A^+$ and $A^-$ respectively. }}The chart $T=0$
 of $\displaystyle\mathbb{P}^3(\mathbb{C})$
is parametrized by $\displaystyle (\frac{y}{x}, \frac{z}{x}, \frac{1}{x})$. Recall that $\pi(A^+) = p^+ = [1,0,0,0]$ is located at the origin of this chart. One can check that a neighborhood of $A^+$, centered at $A^+$, is parametrized by $(y, z/x, 1/x)$, and deduce that $\{|y|<\varepsilon, |z/x|<\varepsilon\}\subset V_{A^+, \varepsilon}$ is a neighborhood of $A^+$. We prove exactly by the same arguments that $V_{A^-, \varepsilon}$ is a neighborhood of $A^-$.\\
In order to simplify notations, let us denote $\displaystyle (x_0,y_0,z_0)\underset{\varepsilon}{\approx} (x_1,y_1,z_1)$ if there exist $\delta_x,\delta_y,\delta_z\in ]-\varepsilon, \varepsilon[$ such that $\displaystyle (x_0,y_0,z_0) = ((1+\delta_x)x_1,(1+\delta_y)y_1,(1+\delta_z)z_1)$. Then it is straightforward that $\forall (x_0, y_0, z_0)\in V_{A^+, \varepsilon}$,\begin{equation}\label{equiv1}\displaystyle f(x_0, y_0, z_0)\underset{\varepsilon}{\approx} (y_0, z_0, dx_0)\in V_{A^-, \frac{\varepsilon}{d}}.\end{equation}

Hence, denoting by $\displaystyle\|\cdot\|$ the max norm, $$\displaystyle |d-\varepsilon| \|p\|_+ \leq \|f^3(p)\|\leq |d+\varepsilon| \|p\|_+.$$Thanks to Proposition \ref{pro.partition}, we can iterate these inequalities in order to obtain the inequalities of the statement, for $\displaystyle p\in W'\cap ( V_{A^+, \varepsilon}\cup  V_{A^-, \varepsilon} \cup  V_{C, \varepsilon}).$\qed

\end{proof} 

\begin{remark}\label{epsilon}
Let $\varepsilon'\, \varepsilon'' \in ]0,\varepsilon[$. By the proof of the previous lemma,  $\{|z|<\varepsilon' |dx|, |y|<\varepsilon''\}$ extends in $X$ in a neighborhood of $A^+$. Then Equation (\ref{equiv1}) implies that \begin{equation}\label{first}\displaystyle f\left(\{|z|<\varepsilon' |dx|, |y|<\varepsilon''\}\cap J\right)\subset\{|x|<\varepsilon'', |y|<\varepsilon'|dz|\}\cap J.\end{equation}

\noindent On the other hand \begin{equation}\label{later}f\left(\{ |z|<\varepsilon' |dx|, |y|<\varepsilon''\}\cap J\right)\subset \{|x|<\varepsilon |dz|, |y|<\varepsilon\}\cap J\end{equation} for $\varepsilon', \varepsilon''$ sufficiently small, since the right hand set in Equation (\ref{later}) is a neighborhood of $A^-$ in $X$. Then Equations (\ref{first}) and (\ref{later}) give $$\displaystyle f\left(\{ |z|<\varepsilon' |dx|, |y|<\varepsilon'\}\cap J\right)\subset\{|x|<\varepsilon'', \text{ }|y|<\varepsilon\}\cap J.$$ We conclude that  $\displaystyle \{|x|<\varepsilon, \text{ }|y|<\varepsilon'\}\cap J\neq \emptyset$. In the same way, $\displaystyle \{|z|<\varepsilon, \text{ }|y|<\varepsilon'\}\cap J\neq \emptyset$. 
\end{remark}


\subsection{The measure $\mu$ is canonical} 

We push our analysis further.  Let us denote by $\log^A$ a function defined on $J$ by $\log^A:=\max(A,\log)$ on $W'\cup K$, for some constant $A\in\mathbb{R}$ sufficiently large so that $K\subset U:=\{\log^A\parallel z\parallel =A\}\cap J$, and by $\log^A:=A$ on $W''$. We have

 \begin{theorem}\label{th.canonique}
Suppose $|d|>1$. Then
$$
\mu  =  \dis\lim_{n\rightarrow +\infty} T\wedge dd^c{\log}^A \parallel f^{3n} \parallel=  \dis\lim_{n\rightarrow +\infty}T\wedge dd^c\left(\frac{1}{n}\sum_{i=0}^{n-1}{\log}^A\parallel f^{i}\parallel \right)
$$
in the weak sense of currents.
\end{theorem} 

\begin{proof}
Let us show the first equality.
\smallskip

\noindent {\bf Preliminaries.} Define on $J$ for all $n\geq 0$
\begin{equation}\label{eq.quasidecreasing}
\dis\phi'_n:=\log^A\parallel f^{3n}\parallel-\sum_{k=1}^n\log \left(|d|+\frac{\varepsilon}{(|d|-\varepsilon)^{k-1}}\right).
\end{equation}

\noindent {\bf The sequence $(\phi'_n)$ is decreasing near infinity.} It results from Lemma \ref{below} that we have on $J$ near infinity, for all $n\geq 1$, 
$$
\dis {\parallel f^{3n}\parallel}\leq {\parallel f^{3(n-1)}\parallel}\cdot\left(|d| + \frac{\varepsilon}{(|d|-\varepsilon)^{3(n-1)}}\right).
$$
Thus for all $n\geq 1$, 
$$
\dis\log^A\parallel f^{3n}\parallel-\log \left(|d|+\frac{\varepsilon}{(|d|-\varepsilon)^{n-1}}\right)\leq \log^A\parallel f^{3(n-1)}\parallel.
$$
\noindent Therefore, the sequence $(\phi'_n)$ is decreasing on $J$ near infinity. (Note that since infinity is attracting on $W$, for all compact $Q$ in $W$ there exists $N\geq 0$ such as $(\phi'_n)$ is decreasing on $Q$ for $n\geq N$). 

\noindent {\bf The sequence $(\phi'_n)$ is decreasing everywhere.} The sequence $(\phi'_n)$ is decreasing outside $V:=\displaystyle f^{-1}(U)\subset U$. We claim that the sequence $\phi'_n$ is in fact decreasing on $W'$ (therefore on $J$).\

\noindent Indeed, take $x\in W'$ and $n\in \mathbb{N}$. If $f^n(x)\in V$, then $\log^A|| f^{n+1}(x) ||= \log^A|| f^n(x) ||=0$, therefore $\phi'_{n+1}<\phi'_n$. If $x\notin V$, then $\phi'_{n+1}(x)\leq \phi'_n(x)$. 

As in the proof of Theorem \ref{measure} we set $\psi_n:=\log^+|| f^{3n}||-n\log |d|$.
\smallskip

\noindent {\bf The sequence $\psi_n$ is quasi-decreasing.} We have seen in Theorem \ref{measure} that the difference $\phi'_n-\psi_n$ on $J$ is uniformly bounded. Thus there exists a constant $c$ such that
$$
\dis\lim_{n\rightarrow +\infty}\phi'_n={\limsup_{n\rightarrow +\infty}}^*\phi'_n={\limsup_{n\rightarrow +\infty}}^*\psi_n+c=\phi_{\infty}+c.
$$
Therefore, when $n\rightarrow +\infty$,
$$
T\wedge dd^c {\log}^+\parallel f^{3n}\parallel=T\wedge dd^c \phi'_n \rightarrow T\wedge dd^c \left(\phi_{\infty}+c\right)=\mu.
$$
\flushright \qed

 
\end{proof}
\begin{corollary}
The support of the measure $\mu$ is included in $\overline{W}\cap K$.
\end{corollary}

\begin{proof} With the previous notations, for all $n\geq 0$, 
$$
\dis\text{supp}\left(T\wedge dd^c{\log}^A \parallel f^{3n} \parallel\right)\subset W.
$$
It follows that 
$
\text{supp}\left(\dis\lim_{n\rightarrow +\infty}T\wedge dd^c{\log}^+ \parallel f^{3n} \parallel\right)\subset \overline{W}.
$
Recall finally that $\text{supp}(\mu)\subset K$ (by Theorem \ref{measure}).\qed
\end{proof}

\end{document}